\documentclass{article}
\usepackage{amsthm}
\usepackage{amsmath}
\usepackage[all]{xy}
\usepackage{amsfonts}
\usepackage{amssymb}
\usepackage{amscd}
\newtheorem{theorem}{Theorem}
\newtheorem{proposition}{Proposition}
\newtheorem{corollary}{Corollary}
\newtheorem{lemma}{Lemma}
\begin{document}
\title{\huge Kummer Generators and Lambda Invariants}    
\author{David Hubbard\thanks{387 Holly Trail,
Crownsville, MD 21032, dhubbard@erols.com}\: and Lawrence C. Washington\thanks{Department of Mathematics,
University of Maryland,
College Park, MD 20742, lcw@math.umd.edu}}
\maketitle

\begin{abstract} Let $F_0=\mathbf Q(\sqrt{-d})$ be an imaginary quadratic field
with $3\nmid d$ and let $K_0=\mathbf Q(\sqrt{3d})$. Let $\varepsilon_0$ be the fundamental unit of $K_0$ and
let $\lambda$ be the Iwasawa $\lambda$-invariant for the cyclotomic $\mathbf Z_3$-extension of $F_0$.
The theory of 3-adic $L$-functions gives conditions for $\lambda\ge 2$ 
in terms of $\epsilon_0$ and the class numbers
of $F_0$ and $K_0$. We construct units of $K_1$, the first level of the 
$\mathbf Z_3$-extension
of $K_0$, that potentially occur as Kummer generators of unramified extensions of
$F_1(\zeta_3)$ and which give an algebraic interpretation of the condition that $\lambda\ge 2$.
We also discuss similar results on $\lambda\ge 2$ that arise from work of Gross-Koblitz.
\end{abstract}

Let $F_0=\mathbf Q(\sqrt{-d})$ be an imaginary quadratic field and let $h^-$ be its class number.
Let $\varepsilon_0$ and $h^+$ be the fundamental unit and class number of $K_0=\mathbf Q(\sqrt{3d})$.
The starting point of this work is the following observation.
If 3 splits in $F_0$, and $h^-$, $h^+$, and $\varepsilon_0$ satisfy certain
congruences, then the theory of $3$-adic $L$-functions shows that
the Iwasawa $\lambda$-invariant for the cyclotomic $\mathbf Z_3$-extension
$F_{\infty}/F_0$ satisfies $\lambda\ge 2$. This means that,
as we go up the tower of fields in the $\mathbf Z_3$-extension,
eventually the 3-Sylow subgroup of the ideal class group
has rank at least 2. Therefore, there are at least two independent
unramified abelian extensions of degree 3. We show in Section 7 that,
after we adjoin a cube root of unity,
$\varepsilon_0$ gives a Kummer generator for one of these unramified
extensions. The question is then, where is the other Kummer generator?
Since a condition on $\varepsilon_0$ is used to show this generator exists,
the goal is to relate this generator to $\varepsilon_0$.
This we do in Sections 6 and 8. 

In fact, inspired by ideas from genus theory, in Section 6
we find special
generators for the part of the units where the potential Kummer generators 
live. The congruence condition on $\varepsilon_0$ can then
be interpreted in terms of these generators.
This gives an algebraic interpretation
of the condition for $\lambda\ge 2$.

In Section 9, we use a result of Gross-Koblitz to obtain similar results.
There is a difference in this approach: In Section 9, we use the
value of the $3$-adic $L$-function only at $s=0$, and the proofs 
in that section involve (with one minor exception) only the minus
parts of the class groups. In the other sections, we use the values
of the $3$-adic $L$-functions at both $s=0$ and $s=1$. These correspond
to the minus components and the plus components of the class groups,
respectively, and the interplay between these two components
naturally leads to the study of Kummer pairings and Kummer generators. 

Another motivation for this work is the following.
We know that $\lambda$ is at least as large as the 3-rank of the class
group of $F_0$.
In tables of $\lambda$ (for example, \cite{gold}), sometimes
$\lambda$ is larger than this rank, so there is a jump
in the 3-rank of the class group as we proceed up the $\mathbf Z_3$-extension.
For example, consider the special case where
$3\nmid h^+$. Scholz's theorem says that the 3-class group of $F_0$ is
trivial or cyclic. In the case that 3 splits in $F_0$, we always
have $\lambda\ge 1$, and we show that in this case
$\varepsilon_0$ is always a Kummer
generator for an unramified extension of $F_1(\zeta_3)$. 
When the 3-class group of $F_0$
is trivial and $\lambda=1$, this explains the jump in the rank of
the class group. When the 3-class group is non-trivial,
then our condition for when $\lambda\ge 2$
gives a condition for when the rank jumps in terms of
a congruence on $\varepsilon_0$. 

\section{ $3$-adic $L$-functions}

Let $\chi$ be the quadratic character attached to $K_0$ and let $L_3(s, \chi)$
be the 3-adic $L$-function for $\chi$. There is a power series $f(T)=f(T,\chi)=a_0+a_1T+\cdots\in \mathbf Z_3[[T]]$
such that 
$$
L_3(s,\chi)=f((1+3)^s-1).
$$
Therefore,
\begin{gather*}
a_0=f(0)=L_3(0,\chi)=\left(1-\chi\omega^{-1}(3)\right)h^-\\
\left(1-\frac{\chi(3)}{3}\right)\frac{2h^+\log_3 \varepsilon_0}{\sqrt{D}}=L_3(1,\chi)=f(3)=a_0+3a_1+\cdots,
\end{gather*}
where $\omega$ is the character for $\mathbf Q(\sqrt{-3})$ and $D$ is the discriminant of $K_0$.
Also, $\lambda$ is the smallest index $i$ such that $a_i\not\equiv 0\pmod 3$. In particular, $\lambda\ge 2$ if and only if
$a_0\equiv a_1\equiv 0\pmod 3$. 

Assume from now on that 3 ramifies in $K_0$. Then $3|D$ and $\chi(3)=0$.
Also, since $3|D$, we have $\text{Norm}(\varepsilon_0)=+1$. The above equations yield
$$
\left(1-\chi\omega^{-1}(3)\right)h^- + 3a_1\equiv \frac{2h^+\log_3 \varepsilon_0}{\sqrt{D}} \pmod 9.
$$

\begin{proposition}\label{padicLprop}
(a) Suppose that 3 splits in $F_0$ and $3\nmid h^+$. Then  
$$\lambda\ge 2\iff \log_3 \varepsilon_0 \equiv 0\pmod 9.$$
(b) Suppose that 3 splits in $F_0$ and $3\mid h^+$. Then $\lambda\ge 2$.\newline
(c) Suppose that 3 is inert in $F_0$ and that $3|h^-$. Then  
$$\lambda\ge 2\iff h^-\equiv \frac{h^+\log_3 \varepsilon_0 }{\sqrt{D}} \pmod 9.$$
\end{proposition}
\begin{proof} If 3 splits, then
$\chi\omega^{-1}(3)=1$, so $a_0=0$. If 3 is inert, then $\chi\omega^{-1}(3)=-1$,
so $a_0=2h^-$. Our assumptions imply that $a_0\equiv 0\pmod 3$.
The condition in part (c) is clearly equivalent to $a_1\equiv 0\pmod 3$.

It follows from the $3$-adic class number formula that
$(\log_3\varepsilon_0)/\sqrt{D}\in\mathbf Q_3$. Therefore, $\log_3\varepsilon_0\equiv 0\pmod 9$
is equivalent to $\log_3\varepsilon_0\equiv 0\pmod {9\sqrt{D}}$. Hence, the condition
in (a) is equivalent to $a_1\equiv 0\pmod 3$.

If $d\equiv 2\pmod 3$, write $\varepsilon_0^2=1+3x+y\sqrt{3d}$ with $2x, 2y\in \mathbf Z$. 
Then
\begin{align*}
\log_3 \varepsilon_0&\equiv \frac12 \log(1+3x+y\sqrt{3d})\\
&\equiv 
y\sqrt{3d}+\frac13 y^3 3d\sqrt{3d}
\equiv 0\pmod 3
\end{align*}
since $d\equiv -1\pmod 3$. Therefore, the condition in (b) implies that
$a_1\equiv 0\pmod 3$. 
\end{proof}

{\bf Remark.} For a slightly different version of the
proposition, see \cite[Thm. 1]{kraft}. 

\section{Preliminaries}

Let $L_0=\mathbf Q(\sqrt{-d},\sqrt{3d})$.
Let $\mathbf B_i$
be the $i$th level of the $\mathbf Z_3$-extension of $\mathbf Q$.
Define $F_i=F_0\mathbf B_i$, $K_i=K_0\mathbf B_i$, and $L_i=L_0\mathbf B_i$. Let
$$
\langle g\rangle = \text{Gal}(K_i/\mathbf B_i)=\text{Gal}(L_i/F_i), 
\quad \langle \sigma\rangle = \text{Gal}(L_i/K_i) =\text{Gal}(F_i/\mathbf B_i),
$$
where we identify the groups for arbitrary $i$ with those for $i=0$.
Let 
$$
\langle \tau \rangle = \text{Gal}(\mathbf B_1/\mathbf B_0)=\text{Gal}(L_1/L_0)=
\text{Gal}(F_1/F_0)
=\text{Gal}(K_1/K_0).
$$
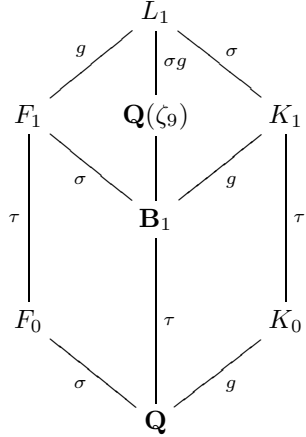
\begin{align*}
&\xymatrix{
&L_1 \ar@{-}[d]^{\sigma g} \ar@{-}[dl]_g \ar@{-}[dr]^{\sigma}\\
F_1\ar@{-}[dr]_{\sigma} \ar@{-}[dd]_{\tau}& {\mathbf Q}(\zeta_9)\ar@{-}[d] & 
K_1\ar@{-}[dl]^g\ar@{-}[dd]^{\tau}\\
& {\mathbf B}_1\ar@{-}[dd]^{\tau}& \\
F_0\ar@{-}[dr]_{\sigma} & & K_0\ar@{-}[dl]^{g}\\
& {\mathbf Q}
}
\end{align*}

Let $A_n$ be the 3-Sylow subgroup of the ideal class group of $F_n$ and let
$\widetilde{A}_n$ be the 3-Sylow subgroup of the ideal class group of $L_n$.

\begin{lemma}\label{scholzlemma} Let $\widetilde{A}_n^{-}$ denote the subgroup
of $\widetilde{A}_n$ on which $\sigma$ acts by inversion.
Then $A_n\simeq \widetilde{A}_n^-$ for all $n\ge 0$. If $3\nmid h^+$ then $A_n\simeq 
\widetilde{A}_n$ for all $n\ge 0$.
\end{lemma}
\begin{proof} 
Let $I$ be an ideal of $F_n$ that becomes principal in $L_n$. Then the norm
from $L_n$ to $F_n$ of this ideal, namely $I^2$, is principal. If $I$ has order
divisible by 3, this is impossible.
Therefore, the natural map from $A_n$ to $\widetilde{A}_n$ is injective.
Now let $J$ represent an ideal class in $\widetilde{A}_n$.
Then 
$$
J^4=J^{(1+g+\sigma+g\sigma)} J^{(1+g)(1-\sigma)} J^{(1+\sigma)(1-g)} 
J^{(1+g\sigma)(1-\sigma)}.
$$
All four ideals on the right-hand side are of 3-power order in the ideal class group.
The first ideal is the lift of an ideal of $\mathbf B_n$, hence principal. The ideal
$J^{1+g\sigma}$ come from $\mathbf Q(\zeta_{3^{n+1}})$,
which has class number prime to 3. Hence this ideal is principal.
The ideal $J^{1+g}$ comes from
$F_n$ and $\sigma$ acts by inversion on $J^{(1+g)(1-\sigma)}$. 

If $3\nmid h^+$ then the class number of $K_n$ is prime to 3, so
$J^{(1+\sigma)(1-g)}$ is principal. 
Therefore, the ideal class of $J^4$ comes from $F_n$. This implies
that the ideal class of $J$ comes from $F_n$.

In general, $\sigma$ fixes $J^{(1+\sigma)(1-g)}$.
If $J$ represents a class in $\widetilde{A}_n^-$, then 
$\sigma$ inverts the class of $J^4/J^{(1+g)(1-\sigma)}$, which is the class
of $J^{(1+\sigma)(1-g)}$. It follows easily
that $J^{(1+\sigma)(1-g)}$ is principal, so the class of $J^4$ comes from
$F_n$, as desired.
\end{proof}

\begin{lemma}\label{ranklemma} The 3-rank of $A_n$ is less than or 
equal to $\lambda$ for all $n$.
\end{lemma}
\begin{proof} Since $F_m/F_n$ is totally ramified for all $m\ge n$, the norm map
on the ideal class groups is surjective. Therefore,
$A_n$ is a quotient of $X=\lim_{\leftarrow} A_m$. Since the 3-class group 
of $\mathbf B_m$ is trivial
for all $m$, the minus part of the 3-class group of $F_m$ is $A_m$. Therefore, by
\cite[Corollary 13.29]{washington},
$X\simeq \mathbb Z_3^{\lambda}$. The result follows easily.\end{proof}

The following lemmas will be useful throughout.

\begin{lemma}\label{quadlemma} Let $L/K$ be a quadratic extension of number fields and let $\sigma$ generate
Gal$(L/K)$. \newline
(a) The map $K^{\times}/(K^{\times})^3 \to \left(L^{\times}/(L^{\times})^3\right)^{\langle \sigma\rangle}$
is an isomorphism (where the superscript $\langle\sigma\rangle$ denotes the
elements fixed by $\sigma$).\newline
(b) Let $E_K, E_L$ be the unit groups of the rings of integers of $K$ and $L$.
The map 
$E_K/(E_K)^3 \to \left(E_L/(E_L)^3\right)^{\langle \sigma\rangle}$
is an isomorphism.
\end{lemma}
\begin{proof}
The cohomology sequence associated with 
$$1\to (L^{\times})^3\to L^{\times}\to 
L^{\times}/(L^{\times})^3\to 1$$
yields
$$ K^{\times} \to \left( L^{\times}/(L^{\times})^3\right)^{\langle \sigma\rangle}
\to H^1(\langle \sigma \rangle,\, (L^{\times})^3).
$$
Since $\sigma$ has order $2$, the cohomology group $H^1(\langle \sigma \rangle,\, 
(L^{\times})^3)$ is killed by 2. 
But the image of $\left( L^{\times}/(L^{\times})^3\right)^{\langle \sigma\rangle}$ in this cohomology group is also a quotient of
the group $\left( L^{\times}/(L^{\times})^3\right)^{\langle \sigma\rangle}$ 
of exponent 3. Therefore, the image is trivial, which yields the surjectivity in part
(a). Now suppose that $x\in K^{\times}$ is a cube in $L^{\times}$. Then $x^2$ is the norm
of a cube, hence is a cube. This implies that $x$ is also a cube, so the map in part (a)
in injective.

 The proof of (b) is similar.
\end{proof}

\begin{lemma}\label{cubelemma} The natural map $K_0^{\times}/(K_0^{\times})^3\to 
K_1^{\times}/(K_1^{\times})^3$ is injective.
\end{lemma}
\begin{proof}
Let $b\in K_0^{\times}$ and suppose that $X^3-b$ has a root in $K_1$.
If $b\not\in (K_0^{\times})^3$, then the polynomial is irreducible in $K_0[X]$.
Since $K_1/K_0$ is Galois, the polynomial has all three roots in $K_1$,
hence $K_1$ contains the cube roots of unity. Contradiction.
\end{proof}

\begin{lemma}\label{ep0lemma} $\varepsilon_0$ is not a cube in $L_1$.
\end{lemma}
\begin{proof}
This follows from the previous two lemmas.\end{proof}

\section{Structure of units}

Our goal is to understand the Hilbert 3-class field of $F_1$
and to lay the groundwork for the proof of Theorem~\ref{mainthm} in Section 6.
The main part of the proof relies on an analysis of the Kummer generators
of unramified extensions of $L_1$. In our situation, these Kummer generators
arise from units in $K_1$.

Let $E$ denote the units of $K_1$. 
By the Dirichlet unit theorem, $E/E^3\simeq (\mathbf Z/3\mathbf Z)^5$.
We have
$$
E/E^3 = (E/E^3)^{1+g}\oplus (E/E^3)^{1-g},
$$
where the groups $(E/E^3)^{1\pm g}$ are the $\pm$ eigenspaces for the action of $g$.
Lemma 3, or the fact that $1+g$ is the norm from $K_1$ to $\mathbf B_1$, 
shows that $(E/E^3)^{1+g}$ is represented by the units of $\mathbf B_1$
and has dimension 2 over $\mathbf F_3$. Therefore $(E/E^3)^{1-g}$ has dimension 3.

Let 
$$M=(E/E^3)^{1-g}.$$ 
and let $\widetilde{H}_3=L_1(M^{1/3})$. Then 
$\widetilde{\cal A}= \text{Gal}(\widetilde{H}_3/L_1)
\simeq (\mathbf Z/3\mathbf Z)^3$.
The Kummer pairing
$$
M\times \widetilde{\cal A} \longrightarrow \mu_3
$$
is Galois equivariant for the actions of $\text{Gal}(L_1/\mathbf Q)$
on $M$, $\widetilde{\cal A}$, and $\mu_3$. 

Since $g$ acts as $-1$ on $\mu_3$ and acts as $-1$ on $M$, it acts
as $+1$ on $\widetilde{\cal A}$. 
In particular, any lift of $g$ to an element of
$\text{Gal}(\widetilde{H}_3/F_1)$ commutes
with $\text{Gal}(\widetilde{H}_3/L_1)$. Since 
$\text{Gal}(\widetilde{H}_3/F_1)$ is therefore abelian 
of order $2\times 27$, we can
lift $g$ to the unique element of order 2 in $\text{Gal}(\widetilde{H}_3/F_1)$. 
We continue to call this element $g$.
The fixed field of $g$ restricted to $\widetilde{H}_3$ is a field $H_3$ 
that is Galois over $F_1$.
Of course, if $\widetilde{H}_3/L_1$ is unramified, 
then this is simply a consequence of the fact
that the Hilbert 3-class field of $F_1$ lifts to the minus part of the Hilbert 3-class field of $L_1$. 

Let ${\cal A}=\text{Gal}(H_3/F_1)$. We have
$$
{\cal A}\simeq \widetilde{\cal A},
$$
and we can rewrite the Kummer pairing as
$$
M\times {\cal A}\longrightarrow \mu_3.
$$
In many ways, it is more natural to state the results over $F_1$. But,
for the proofs, it is often necessary
to work in the larger field $L_1$ in order to use Kummer generators.

Since $\tau^3=1$, we find that $(1-\tau)^3$ kills $M$. There are two
filtrations: 
$$
M=M^0\supseteq M^1=(1-\tau)M\supseteq M^2=(1-\tau)^2 M\supseteq M^3=0
$$
and
$$
M=M_3\supseteq M_2=M[(1-\tau)^2]\supseteq M_1=M[1-\tau]\supseteq M_0=0,
$$
where $M[(1-\tau)^j]$ denotes the kernel of $(1-\tau)^j$ (cf. \cite{hubbard}).
Moreover,
$$
(1-\tau)^{3-j}M\subseteq M[(1-\tau)^j]
$$
for $0\le j\le 3$. 

There are also two
filtrations on ${\cal A}$: 
$$
{\cal A}={\cal A}^0\supseteq {\cal A}^1=(1-\tau){\cal A}\supseteq {\cal A}^2=(1-\tau)^2 {\cal A}\supseteq {\cal A}^3=0
$$
and
$$
{\cal A}={\cal A}_3\supseteq {\cal A}_2={\cal A}[(1-\tau)^2]\supseteq {\cal A}_1={\cal A}[1-\tau]\supseteq {\cal A}_0=0,
$$
where ${\cal A}[(1-\tau)^j]$ denotes the kernel of $(1-\tau)^j$.
Moreover,
$$
(1-\tau)^{3-j}{\cal A}\subseteq {\cal A}[(1-\tau)^j]
$$
for $0\le j\le 3$. 

Under the Kummer pairing $M\times {\cal A} \to \mu_3$,
we have 
\begin{align*}
a\in (M^i)^{\perp}&\iff \langle (1-\tau)^i m, a\rangle=1 \text{ for all }m\in M\\
&\iff\langle m, (1-\tau^{-1})^ia\rangle=1\text{ for all }m\in M\\
&\iff(1-\tau)^i a=0,
\end{align*}
so $(M^i)^{\perp}= {\cal A}_i$. This and similar facts yield nondegenerate
pairings
\begin{gather*}
M^{i}\times {\cal A}/{\cal A}_i\to \mu_3,
\qquad M/M^{i}\times {\cal A}_{i}\to \mu_3,\\ 
M_{i}\times {\cal A}/{\cal A}^i\to \mu_3,
\qquad M/M_{i}\times {\cal A}^{i}\to \mu_3, 
\end{gather*}
so each filtration for $M$ pairs with a filtration for ${\cal A}$.

Lemma 5 (or 4) says that $\varepsilon_0$ gives a nontrivial element of $M[1-\tau]$. 

\begin{lemma}\label{dimlemma} The following are equivalent (where $\dim$ is the dimension
as a vector space over $\mathbf F_3$):\newline
(a) $\dim M_1=1$\newline
(b) $\dim M_i=i$ for all $i$\newline
(c) $M_i=M^{3-i}$ for all $i$ \newline
(d) $\dim M^i=3-i$ for all $i$\newline
(e) $\dim {\cal A}_i=i$ for all $i$\newline
(f) ${\cal A}_i={\cal A}^{3-i}$ for all $i$\newline
(g) $\dim {\cal A}^i=3-i$ for all $i$.\newline
(h) $\dim {\cal A}_1=1$.
\end{lemma}
\begin{proof} Regard $\tau$ as a linear transformation of $M$.
The characteristic polynomial of $\tau$ is $T^3-1=(T-1)^3$.
The equivalence of (a), (b), (c), (d) follows easily from a 
consideration of the three possibilities
for the Jordan canonical form of $\tau$. 
The equivalence of (e), (f), (g), and (h) follows similarly.

The equivalence of (b) and (g) follows from the duality between
$M_i$ and ${\cal A}/{\cal A}^i$.
\end{proof}

All the standard questions of genus theory for an odd prime 
can be asked about
${\cal A}$ and $M$.  For
instance, let $H_i$ be the fixed field of ${\cal A}^i$.  Then $H_1$ is the ``genus
subfield'' of $H_3$; namely,
$H_1$ is the maximal abelian extension of $F_0$ contained in $H_3$.  Equivalently,
$H_1$ is maximal
with $\tau$ acting trivially on $\text{Gal}(H_1/F_1)$.  
Thus we can call ${\cal A}/{\cal A}^1$ the
``genus group'' for ${\cal A}$.

  We can also define $M/M^1$ to be the ``genus unit group,'' 
as it is the maximal quotient of $M$
with trivial action by $\tau$.  Similarly, we can call ${\cal A}_1$ 
and $M_1$ the ambiguous subgroups.  Then
the nondegenerate pairing
$$M_1\times{\cal A}/{\cal A}^1\to\mu_3$$
says that the ambiguous unit group pairs nontrivially with the genus group.  
A similar statement can be made for
$$M/M^1\times{\cal A}_1\to\mu_3.$$

Thus, ambiguous unit classes are Kummer generators for the 
genus field $H_1$.  In general, the
duality between $M_i$ and ${\cal A}/{\cal A}^i$ says that $M_i$ 
is the group of Kummer generators for the
fixed field of ${\cal A}^i$.

Suppose now that $\dim M_i=i$ for all $i$.  
Then there are units $u_1$, $u_2$, and $u_3$ such
that $u_i\in M_i$ but $u_i\not\in M_{i-1}$, and $\{u_1,u_2,u_3\}$ is a
basis for $M$.  The fixed field
of ${\cal A}^i$ is then $L_1(u_1^{1/3},\ldots,u_i^{1/3})$.

We can extend the definition of genus and ambiguous groups. We call
${\cal A}^i/{\cal A}^{i+1}$ and $M^i/M^{i+1}$
the $i$th higher genus group and we call ${\cal A}_{i+1}/{\cal A}_i$ and
$M_{i+1}/M_i$ the $i$th higher
ambiguous groups.  These are all maximal quotients with trivial $\tau$ action. 
There are also nondegenerate pairings
$$M^i/M^{i+1}\times{\cal A}_{i+1}/{\cal A}_i\to\mu_3$$
and
$$M_{i+1}/M_i\times{\cal A}^i/{\cal A}^{i+1}\to\mu_3.$$
See [5], section 2 for a use of some of these groups.

\section{Capitulation}

Let $A(K_i)$ be the Sylow 3-subgroup of the class group of $K_i$.
There is a natural homomorphism $A(K_0)\to A(K_1)$ that maps
the class of an ideal of $K_0$ to the class of the ideal
generated by that ideal in $K_1$. 
\begin{proposition}\label{caplemma} 
There is an isomorphism $M[1-\tau]/\langle \varepsilon_0
\rangle \simeq \text{Ker}(A(K_0)\to A(K_1))$.
\end{proposition}
\begin{proof} Let $\varepsilon \in M[1-\tau]$. Then $\varepsilon^{1-\tau}
=\gamma^3$ for some $\gamma\in E$. Taking the norm to $K_0$ yields
$\text{Norm}(\gamma)^3=1$. Since $\zeta_3\not\in K_0$, we have
$\text{Norm}(\gamma)=1$. By Hilbert's Theorem 90, $\gamma=\eta^{1-\tau}$
for some $\eta\in K_1$. Then $\alpha=\varepsilon/\eta^3$ is fixed by $\tau$,
so it lies in $K_0$. Let $J=(\eta^{-1})$. Then $J^3=(\alpha)=(\tau\alpha)=(\tau J)^3$, 
so $J$ is
fixed by $\tau$ (alternatively, $J^{1-\tau}=(\gamma)=(1)$). Therefore, $J=I\mathfrak p_1^a$ for some integer $a$,
where $I$ comes from $K_0$ and $\mathfrak p_1$ is the prime of $K_1$
above $3$. This implies that $I^{1-g}=J^{1-g}=(\eta^{g-1})$.
Therefore, $I^{1-g}\in \text{Ker}(A(K_0)\to A(K_1))$. 
Define
\begin{align*}
\psi: M[1-\tau] &\longrightarrow \text{Ker}(A(K_0)\to A(K_1))\\
\varepsilon \qquad&\longmapsto \qquad I^{1-g}.
\end{align*}
It is straightforward to check that the ideal class of $I^{1-g}$ is independent of
the various choices made and depends only on $\varepsilon$ mod cubes. Therefore,
$\psi$ is well-defined.

If $\varepsilon=\varepsilon_0$, then we may take $\gamma=\eta=1$,
so $\alpha=\varepsilon_0$ and $J$ and $I^{1-g}$ are trivial. Therefore,
$\varepsilon_0\in \text{Ker}(\psi)$.

Suppose $\psi(\varepsilon)=1$. Then $I^{1-g}=(\delta)$ for some
$\delta\in K_0$. Then 
$$
(\delta^3)= I^{3(1-g)} = (\alpha^{1-g}).
$$
Since $\alpha, \delta\in K_0$, we have
$\alpha^{1-g}/\delta^3=\pm\varepsilon_0^b$ for some $b$. Therefore,
$\varepsilon^{1-g}\equiv\alpha^{1-g}\equiv \varepsilon_0^b$ mod cubes.
Since $\varepsilon \in M=(E/E^3)^{1-g}$, we have
$\varepsilon^{1-g}\equiv \varepsilon^2$ mod cubes. Therefore,
$\varepsilon^2$, and hence $\varepsilon$, is in $\langle\varepsilon_0\rangle$
mod cubes. Consequently,
the kernel of $\psi$ is $\langle\varepsilon_0\rangle$.

Now let $B$ be an ideal in $\text{Ker}(A(K_0)\to A(K_1))$. 
Let $B=(y)$ with $y\in K_1$. Then $y^{1-\tau}\in E$.
Since the norm from $K_1$ to $K_0$ of a principal ideal is
principal, we have $B^3=(x)$ with $x\in K_0$. 
Then $y^3=ux$ for a unit $u\in E$. But $u=x^{-1}y^3$,
so $u^{1-g}\in M[1-\tau]$. Since $u^{1-\tau}=(y^3)^{1-\tau}$,
we take $\gamma=(y^{1-g})^{1-\tau}$ and $\eta=y^{1-g}$. Then
$(\eta^{1-g})=(y)^{2(1-g)}=B^{2(1-g)}$, so $\psi(u^{1-g})=B^{2(1-g)}$.
But the ideal class of $B^{g}$ is the ideal class of $B^{-1}$, and the ideal
class of $B$ has order 3 since it capitulates, so $\psi(u^{1-g})=B$.
Therefore, $\psi$ is surjective.\end{proof}

\begin{proposition}\label{oneorthreeprop} The kernel of the map $A(K_0)\to A(K_1)$
has order 1 or 3. 
\end{proposition}
\begin{proof}
Let $I$ be an ideal of $K_0$ that becomes principal in $K_1$.
Say $I=(y)$ with $y\in K_1$. Then $y^{1-\tau}\in E$. It is easy
to see that the map $\rho: I\mapsto (\tau\mapsto y^{1-\tau})$
gives an injective homomorphism from $\text{Ker}(A(K_0)\to A(K_1))$
to $H^1(K_1/K_0, E)$ (the image is the group of locally
trivial cohomology classes; see \cite{schmithals}). 

The Herbrand quotient is given by
$$
\frac{\left|\widehat{H}^0(K_1/K_0,E)\right|}{\left|H^1(K_1/K_0, E)\right|}
= \frac13
$$
(see \cite[IX, \S 4, Corollary 2]{lang}),
where $\widehat{H}^0(K_1/K_0, E)$ is the group of 
units of $K_0$
mod norms of units. Since $\varepsilon_0^3=\text{Norm}(\varepsilon_0)$,
the norms of units are either $\pm \langle \varepsilon_0\rangle$
or $\pm \langle \varepsilon_0^3\rangle$. Therefore, $|\widehat{H}^0|=1$ 
or 3. Therefore, $|H^1|=3$ or 9. 

Let $\mathfrak p_1$ be the prime of $K_1$ above $3$. Since $\mathbf B_1$
has class number 1 and $\mathfrak p_1^2$ is the prime of $\mathbf B_1$
above $3$, we have $\mathfrak p_1^2=(\beta)$ for some $\beta\in \mathbf B_1$.
The cocycle $\tau\mapsto \beta^{1-\tau}$ gives a cohomology class
in $H^1(K_1/K_0, E)$. We claim that it is not in the image of $\rho$.
If it is, then there is an ideal $I$ of $K_0$ with $I=(y)$ in $K_1$
and such that $y^{1-\tau}=\beta^{1-\tau}$. Then $y/\beta\in K_0$.
Let $v$ be the $\mathfrak p_1$-adic valuation normalized so that $v(\mathfrak p_1)=1$.
Then $v(\beta)=2$ and $v(y/\beta)=v(I)-2\equiv 1\pmod 3$ since $I$ comes from $K_0$.
This contradicts the fact that $y/\beta\in K_0$. Therefore, the cocycle
is not in the image of $\rho$.
Consequently, the image of $\rho$ has order 1 or 3. Since $\rho$ is injective,
this completes the proof.\end{proof}

\begin{proposition}\label{nontrivprop} $\tau$ does not act trivially on $(E/E^3)^{1-g}$.
That is, $M[1-\tau]$ has dimension 1 or 2. It has dimension 1 if and only if
the map $A(K_0)\to A(K_1)$ is injective.
\end{proposition}
\begin{proof} This follows immediately from the preceding lemma and proposition.
\end{proof}

The following is what we need in subsequent sections.
 
\begin{theorem}\label{keylemma} Exactly one of the following cases holds:
\newline
{\bf Case (i)}\begin{enumerate}
\item[(a)] $M[1-\tau]$ has dimension 1 as an 
$\mathbf F_3$-vector space.
\item[(b)] The map from the class group
of $K_0$ to the class group of $K_1$ is injective.
\item[(c)] The norm map from the units of $K_1$ to the units of $K_0$ is surjective.
\item[(d)] There are units $\varepsilon_1, \varepsilon_2\in K_1$ such that
$\{\varepsilon_0, \varepsilon_1, \varepsilon_2\}$ is a basis for $M$ and
such that
$$
\varepsilon_2^{1-\tau}=\varepsilon_1, \quad \text{ and }\; \varepsilon_2^{1+\tau+\tau^2}=\varepsilon_0.
$$
\end{enumerate}
{\bf Case (ii)}
\begin{enumerate}
\item[(a)] $M[1-\tau]$ has dimension 2 as an 
$\mathbf F_3$-vector space.
\item[(b)] The kernel of the map from the class group
of $K_0$ to the class group of $K_1$ has order 3.
\item[(c)] The cokernel of the norm map from the units of $K_1$ to 
the units of $K_0$ is has order 3.
\item[(d)] There are units $\varepsilon_1, \varepsilon_2\in K_1$ such that
$\{\varepsilon_0, \varepsilon_1, \varepsilon_2\}$ is a basis for $M$ and
such that
$$
\varepsilon_2^{1-\tau}=\varepsilon_1, \quad \text{ and }\; \varepsilon_2^{1+\tau+\tau^2}=1.
$$
\end{enumerate}
\end{theorem}
\begin{proof} Proposition~\ref{nontrivprop} says that exactly one of (i)(a) and (ii)(a)
holds.
Proposition~\ref{caplemma} 
implies that (a) is equivalent to (b) in both cases.

Clearly (d) implies (c) in both cases.

We now prove that (i)(a) implies (i)(d). Assume (i)(a).
By Lemma~\ref{dimlemma}, we have
$$
\dim M[(1-\tau)^j]=\dim (1-\tau)^{3-j}M=j.
$$
In particular, 
there are units $u_1$ and $u_2$ such that
$$
u_2^{1-\tau}=u_1, \quad \text{ and }\; u_1^{1-\tau}=\varepsilon_0\delta^3
$$
for some $\delta \in E$. (Note that we do not need to multiply $u_1$ by a cube
since we can simply modify it by a cube, if necessary.)
This implies that 
$$u_2^{1+\tau+\tau^2}=u_2^{(1-\tau)^2}u_2^{3\tau}=\varepsilon_0\delta_1^3$$
for some unit $\delta_1\in K_1$. Moreover, $\delta_1^3=\varepsilon_0^{-1}u_2^{1+\tau+\tau^2}
\in K_0$. Lemma 4 implies that $\delta_1\in K_0$.
Letting $\varepsilon_2=u_2/\delta_1$ and $\varepsilon_1=\varepsilon_2^{1-\tau}=u_1$,
we have a basis $\{\varepsilon_0, \varepsilon_1, \varepsilon_2\}$ of $M$ such that
$$
\varepsilon_1=\varepsilon_2^{1-\tau}, \quad \varepsilon_0=\varepsilon_2^{1+\tau+\tau^2}.
$$
Therefore, (i)(a) implies (i)(b), (i)(c), and (i)(d).

Now assume (ii)(a). 
Since $M[1-\tau]$ has dimension greater than 1, Lemma~\ref{dimlemma} implies that 
$(1-\tau)^2$ annihilates $M$. Let $\varepsilon\in M$. Then
$$
\varepsilon^{1+\tau+\tau^2}=\varepsilon^{(1-\tau)^2}\varepsilon^{3\tau}=1\in M.
$$
Since the norm maps $E^{1+g}$ to powers of $\varepsilon_0^{1+g}=1$, the map
$1+\tau+\tau^2$ annihilates $E/E^3=(E/E^3)^{1+g}\oplus M$. Therefore,
the norm from units of $K_1$ to units of $K_0$ is not surjective.
Since $\varepsilon_0^3$ is in the image, the cokernel has order 3, which is (ii)(c).

Again assuming (ii)(a), we have units $u_1, u_2, u_3$ that form a basis
of $M$ and such that $u_3^{1-\tau}=u_2$, $u_2^{1-\tau}=\delta_2^3$, 
and $u_1^{1-\tau}=\delta_1^3$
for some units $\delta_i$. 

We claim that $u_2$ is not equivalent to $\varepsilon_0^{\pm 1}$
mod cubes: Suppose that $u_2=\varepsilon^{\pm 1}\delta^3$ for some unit $\delta$.
Applying the norm for $K_1/K_0$ 
to the relation $u_3^{1-\tau}=\varepsilon_0^{\pm 1}\delta^3$
yields $\varepsilon_0^3=\text{Norm}(\delta)^{\mp 3}$. Since $\zeta_3\not\in K_1$,
we obtain $\varepsilon_0=\text{Norm}(\delta^{\mp 1})$, which means that
the norm is surjective. Since (ii)(a) implies (ii)(c),
which implies that the norm is not surjective, we have a contradiction.

Write $\varepsilon_0\equiv u_1^a u_2^b u_3^c$ mod cubes. Applying
$1-\tau$ yields $1\equiv u_2^c$ mod cubes. Therefore, $c\equiv 0\pmod 3$. 
By the claim, we cannot also have $a\equiv 0\pmod 3$. It follows that we
may replace $u_1$ by $\varepsilon_0$ and obtain a basis of $M$ satisfying
the same relations as $u_1, u_2, u_3$. We therefore assume that $u_1=\varepsilon_0$.

Since the image of the norm map has index 3, we have $\text{Norm}(u_2)=\pm
(\varepsilon_0)^{3k}$ for some $k$, and by changing the sign of $u_3$ 
if necessary we may assume that $\pm = +$. Then $\text{Norm}(u_3/\varepsilon_0^k)=1$.
Let
$$
\varepsilon_2=u_3/\varepsilon_0^k, \quad \varepsilon_1= \varepsilon_2^{1-\tau}=u_3^{1-\tau}=u_2.
$$
These are the desired units.

Therefore, (ii)(a) implies (ii)(b), (ii)(c), and (ii)(d).
\end{proof}

Continuing with the analogy with genus theory, we can ask if every ambiguous unit
class in $M_1=M[1-\tau]$ contains an ambiguous unit.  The group 
of ambiguous units is generated
by $\varepsilon_0$.  In case (i), where there is no capitulation from $K_0$ to
$K_1$, the unit $\varepsilon_0$
generates the group of ambiguous unit classes, so every ambiguous unit class contains
an ambiguous unit.  However, in case (ii), when there is capitulation from $K_0$ to $K_1$, 
the dimension of $M_1$ is 2.  It is curious to note that in this case, since
$\varepsilon_0$ generates a
subgroup of dimension 1, there are ambiguous unit classes that do not contain an
ambiguous unit.

\section{Subfields}

Let $\widetilde{H}_i=L_1(\varepsilon_0^{1/3}, \dots, \varepsilon_{i-1}^{1/3})$.
Then $g$ acts as $-1$ on the Kummer generators of $\widetilde{H}_i/L_1$
and acts as $-1$ on $\mu_3$, so $g$ acts trivially on $\text{Gal}(\widetilde{H}_i/L_1)$.
Therefore, $\widetilde{H}_i/F_1$ is abelian and has a unique element
of order 2, which we call $g$. Let $H_i$ be the fixed field of $g$.
Then $H_i/F_1$ is Galois with group $(\mathbf Z/3\mathbf Z)^i$.

We can actually say a lot more:
the extensions $\widetilde{H}_i/L_1$ are lifts of extensions $H_i'/F_0$.
We do not know any applications, but the proof shows
that the result is closely related to the $\tau$-structure of $M$.
\begin{proposition} For $i=1, 2, 3$, there are extensions
$H_i'/F_0$  with $[H_i':F_0]= 3^{i}$ such that
$H_i=F_1H_i'$.
\end{proposition}
\begin{proof} 
We need the following result.
\begin{lemma}
Let $p$ be a prime and let $N_1/N_0$ be a Galois extension
of fields of characteristic not $p$. Assume that $\text{Gal}(N_1/N_0)$
contains an element $\tau$ of order $p$. 
Let $\varepsilon\in N_1^{\times}$ be such that $N_1(\zeta_p,\varepsilon^{1/p})/N_0$ 
is Galois.
Then $\varepsilon^{\tau-1}=\beta^p$ for some $\beta\in N_1$. Moreover,
$\tau$ has an extension to an element of 
$\text{Gal}(N_1(\zeta_p,\varepsilon^{1/p})/N_0)$
of order p if and only if  $\beta^{1+\tau+\dots+\tau^{p-1}}=1$.
\end{lemma}
\begin{proof} 
Take any extension $\widetilde{\tau}$ of $\tau$ to $N_1(\zeta_p,\varepsilon^{1/p})$
that is trivial on $N_0(\zeta_p)$. Note that such extensions exist since
$\tau$ must be trivial on $N_0(\zeta_p)\cap N_1$ because
its degree over $N_0$ is prime to $p$.

Let $N_0'\subseteq N_1$ be the fixed field of $\tau$.  Then 
$N_1(\zeta_p,\varepsilon^{1/p})/N_0'(\zeta_p)$
is Galois of order $p$ or $p^2$ and hence abelian.  Thus $\widetilde{\tau}$ 
commutes with a generator of
$\text{Gal}(N_1(\zeta_p,\varepsilon^{1/p})/N_1(\zeta_p))$, so it follows 
(by a straightforward calculation or by the Kummer pairing) that
$\varepsilon^{\tau-1}=\beta^p$ for some $\beta\in N_1$.

We have
$\widetilde{\tau}(\varepsilon^{1/p})=\varepsilon^{1/p}\beta\zeta$ 
for some (possibly trivial)
$p$th root of unity $\zeta$. 
An easy calculation shows that 
$$
\widetilde{\tau}^p(\varepsilon^{1/p})=\varepsilon^{1/p}
\beta^{1+\tau+\cdots +\tau^{p-1}}\zeta^p
=\beta^{1+\tau+\cdots +\tau^{p-1}}\varepsilon^{1/p}.
$$
Therefore,  if 
$\beta^{1+\tau+\cdots +\tau^{p-1}}=1$, then $\widetilde{\tau}$ has order $p$.
Conversely, if $\tau$ has an extension with order $p$, then this extension
is trivial on $N_0(\zeta_p)$. Therefore, the above calculation shows that
$\beta^{1+\tau+\cdots +\tau^{p-1}}=1$.
\end{proof}

In our situation, the lemma directly implies that $\tau\in\text{Gal}(L_1/L_0)$
yields an element of $\text{Gal}(\widetilde{H}_1/L_0)$ of order $3$. 
But we need to be more explicit. Since 
$\text{Gal}(\widetilde{H}_1/L_0(\varepsilon_0^{1/3}))$
restricts isomorphically to $\text{Gal}(L_1/L_0)$, we choose the element
that restricts to $\tau$ and continue to
call this new element $\tau$. The fixed field of $\tau$ is $L_0(\varepsilon_0^{1/3})$.

Now assume that we are in Case (i) of Theorem 1. 
Apply the lemma to the extension $\widetilde{H}_1/L_0$ with $\varepsilon=\varepsilon_1$.
We have $\varepsilon_1^{\tau-1}= (\varepsilon_0^{-1/3}\varepsilon_2^{\tau})^3$.
Since $\tau(\varepsilon_0^{1/3})=\varepsilon_0^{1/3}$, we have
$$
(\varepsilon_0^{-1/3}\varepsilon_2^{\tau})^{1+\tau+\tau^2}=
\varepsilon_0^{-1}\varepsilon_2^{\tau+\tau^2+1}=1.
$$
The lemma therefore yields an element $\tau\in \text{Gal}(\widetilde{H}_2/L_0)$ of order 3.

Finally, apply the lemma to $\widetilde{H}_2/L_0$ with $\varepsilon=\varepsilon_2$.
We have $\varepsilon_2^{\tau-1} = (\varepsilon_1^{-1/3})^3$.
Since 
$\tau(\varepsilon_1)=\varepsilon_1\varepsilon_0^{-1}\varepsilon_2^{3\tau}$,
we have
$\tau(\varepsilon_1^{1/3})=\zeta \varepsilon_1^{1/3}\varepsilon_0^{-1/3}
\varepsilon_2^{\tau}$ for some $3$rd root of unity $\zeta$. Therefore,
\begin{align*}
&(\varepsilon_1^{1/3})^{1+\tau+\tau^2}\\
&= \left(\varepsilon_1^{1/3}\right)\left(\zeta
\varepsilon_1^{1/3}\varepsilon_0^{-1/3}\varepsilon_2^{\tau}\right)\left(\zeta^2
\varepsilon_1^{1/3}\varepsilon_0^{-1/3}\varepsilon_2^{\tau}\varepsilon_0^{-1/3}
\varepsilon_2^{\tau^2}\right)\\
&=\zeta^3\varepsilon_1\varepsilon_0^{-1}\varepsilon_2^{2\tau+\tau^2}\\
&=\varepsilon_2^{(1-\tau)-(1+\tau+\tau^2)+(2\tau+\tau^2)}=1.
\end{align*}
We obtain $\tau\in \text{Gal}(\widetilde{H}_3/L_0)$ of order 3.

Now assume that we are in Case (ii) of Theorem 1. We have
$$
\varepsilon_1^{\tau-1}=\varepsilon_2^{-(1+\tau+\tau^2)}\varepsilon_2^{3\tau}
=\varepsilon_2^{3\tau},
$$
and $(\varepsilon_2^{\tau})^{1+\tau+\tau^2}=1$. The lemma yields
$\tau\in \text{Gal}(\widetilde{H}_2/L_0)$ of order 3.

Finally, apply the lemma to $\widetilde{H}_2/L_0$ and $\varepsilon_2$. We have
$\varepsilon_2^{\tau-1}=(\varepsilon_1^{1/3})^3$.  We have shown that $\varepsilon_1^{\tau}
=\varepsilon_1\varepsilon_2^{3\tau}$. 
Therefore, $(\varepsilon_1^{1/3})^{\tau}=\zeta\varepsilon_1^{1/3}\varepsilon_2^{\tau}$
for some 3rd root of unity $\zeta$. It follows that
\begin{align*}
(\varepsilon_1^{1/3})^{1+\tau+\tau^2}&=(\varepsilon_1^{1/3})(\zeta\varepsilon_1^{1/3}
\varepsilon_2^{\tau})(\zeta^2\varepsilon_1^{1/3}(\varepsilon_2)^{\tau+\tau^2})\\
&=\varepsilon_1\varepsilon_2^{2\tau+\tau^2}\\
&=\varepsilon_2^{1+\tau+\tau^2}=1.
\end{align*}

Therefore, in both cases, we obtain $\tau\in\text{Gal}(\widetilde{H}_3/L_0)$.
Since $g$ is the unique element of order 2 in the normal subgroup 
$\text{Gal}(\widetilde{H}_3/F_1)$
of $\text{Gal}(\widetilde{H}_3/F_0)$, we must have
$\tau g \tau^{-1} = g$, so $g$ and $\tau$ commute. Therefore,
$\tau g$ has order 6 in $\text{Gal}(\widetilde{H}_3/F_0)$. The fixed field of $\tau g$ is the desired field
$H_3'$. If we restrict $\tau g$ to $\widetilde{H}_2$, its fixed field 
yields $H_2'$ and the restriction to
$\widetilde{H}_1$ yields $H_1'$.
\end{proof}

\section{Kummer generators}

The goal of this section is to prove the following.

\begin{theorem}\label{mainthm} Let $0<d\not\equiv 0\pmod 3$. Let $\varepsilon_0$ be the fundamental 
unit of $\mathbf Q(\sqrt{3d})$. Suppose that
$3\nmid h^+=h(\mathbf Q(\sqrt{3d}))$. Let $A_1$ be the 3-Sylow subgroup of the ideal
class group of $F_1$.
\newline
(a) There are units $\varepsilon_1, \varepsilon_2\in K_1$ such that
$$
\varepsilon_2^{1-\tau}=\varepsilon_1, \quad \text{ and }\; \varepsilon_2^{1+\tau+\tau^2}=\varepsilon_0.
$$
\newline
(b) The 3-rank of $A_1$ is at most 3.
\newline
(c) The 3-rank of $A_1$ is at least 1 if and only if 
$L_1(\varepsilon_0^{1/3})/L_1$ is unramified.
\newline
(d) The 3-rank of $A_1$ is at least 2 if and only if
$L_1(\varepsilon_1^{1/3})/L_1$ is unramified. 
\newline
(e) The 3-rank of $A_1$ is 3 if and only if
$L_1(\varepsilon_2^{1/3})/L_1$ is unramified. 
\end{theorem}

\begin{proof}
Let $H^{(3)}$ be the maximal unramified elementary abelian 3-extension of $L_1$. 
Then Gal$(H^{(3)}/L_1)\simeq \widetilde{A}_1/\widetilde{A}_1^3$.
Since $\zeta_3\in L_1$, we have
$$
H^{(3)}=L_1(V^{1/3}), \quad \text{for some subgroup} \quad V\subset 
L_1^{\times}/(L_1^{\times})^3.
$$
The Kummer pairing
$$
V\times (\widetilde{A}_1/\widetilde{A}_1^3)\longrightarrow \mu_3,
$$
where $\mu_3$ is the group of cube roots of unity, is Galois equivariant.
Since $3\nmid h(K_1)$, we see that $\sigma$ acts on $\widetilde{A}_1/\widetilde{A}_1^3$ 
as $-1$. Also, $\sigma$
acts on $\mu_3$ as $-1$. Therefore, $\sigma$ acts on $V$ as $+1$.
By Lemma~\ref{quadlemma}, we may assume that $V\subset K_1^{\times}/(K_1^{\times})^3$.

By assumption, 3 is unramified in $F_0/\mathbf Q$. Since the extension
$L_1/K_1$ is the lift of the extension $F_0/\mathbf Q$, it is unramified at 3.
Since it is the lift of $\mathbf Q(\sqrt{-3})/\mathbf Q$, it ramifies at most
at 3 and the archimedean primes. 
Therefore, $L_1/K_1$ is unramified at all finite primes. Let $b\in V$.
We also denote the corresponding element of $K_1^{\times}$ by $b$. Therefore,
$\sigma(b)=b$.  Since
$H^{(3)}/L_1$ is unramified, we must have $(b)=I^3$ for some ideal $I$ of $L_1$.
Since $\sigma(I)=I$, and since $L_1/K_1$ is unramified at all finite primes, 
$I$ comes from an ideal
of $K_1$, so $(b)=I^3$ as ideals of $K_1$. Since
$3\nmid h(K_1)$ (because $3\nmid h^+$, and there is exactly one 
ramified prime and it is totally
ramified), the ideal $I$ is principal, so $b=\varepsilon \alpha^3$ for some
$\alpha\in K_1$ and some unit $\varepsilon$ of $K_1$. We have shown that
$V$ is represented by units of $K_1$.

Since the 3-parts of the class groups of $L_1$ and $F_1$ are isomorphic,
$g$ acts as $+1$ on 
$\widetilde{A}_1/\widetilde{A}_1^3$,
so $V\subseteq M=(E/E^3)^{1-g}$.

We are assuming that $3\nmid h^+$, so part (i)(b) of Theorem~\ref{keylemma} holds.
Therefore, $\varepsilon_0, \varepsilon_1, \varepsilon_2$ exist and
$M\simeq \mathbf F_3[T]/(T^3)$, where $1-\tau\leftrightarrow T$. The only subspaces
stable under the action of $T$ are $0$, $(T^2)$, $(T)$, and the whole space.
This means that $V$ is one of $0$, $\langle \varepsilon_0\rangle$, $\langle \varepsilon_0,
\varepsilon_1\rangle$, or $\langle \varepsilon_0, \varepsilon_1, \varepsilon_2\rangle$.
Parts (b), (c), (d), (e) of the theorem follow immediately.
\end{proof}

It is interesting to note that although $\text{Gal}(H_3/F_1)\simeq (\mathbf Z/3\mathbf Z)^3$,
which is very much non-cyclic,
the $\tau$-structure forces it to behave like a cyclic extension of degree 27
in that the inertia subgroup is restricted to be one of four different distinguished
subgroups of orders 1, 3, 9, 27. Namely, there are subgroups
$$
1 \subset (1-\tau)^2{\cal A}\subset (1-\tau) {\cal A}\subset {\cal A},
$$
and the inertia group and the decomposition group
must be among these. Correspondingly, there is a tower of fields 
$$
H_3\supset H_2\supset H_1\supset F_1.
$$
The splitting of primes above 3, if it occurs, happens at the bottom of this tower
and the ramification (possibly trivial) happens at the top.

\section{Kummer generators in the split case}

Let $r$ be the 3-rank of the class group of $K_0$. By Scholz's theorem,
the  3-rank of $A_0$ is $r$ or $r+1$, so $\lambda\ge r$ by Lemma~\ref{ranklemma}.
In fact, we can do better.

\begin{theorem}\label{unramthm} Let $d\equiv 2\pmod 3$ and 
let $\varepsilon_0$ be the fundamental
unit of $K_0=\mathbf Q(\sqrt{3d})$. Let $r$ be the 3-rank of the class group
of $K_0$ and let $A_1$ be the 3-part of the class group of $F_1$.
Then $\text{rank}(A_1)\ge r+1$.
Let $I_1, \dots, I_r$ represent independent ideal classes of order 3 in $K_0$,
and write $I_i^3=(\gamma_i)$ with $\gamma_i\in K_0$. 
Let $L_1=\mathbf Q(\sqrt{-d},\sqrt{3d}, \zeta_9)$.
Then 
$$
L_1(\varepsilon_0^{1/3}, \gamma_1^{1/3}, \dots, \gamma_r^{1/3})/L_1
$$ 
is an everywhere
unramified extension of degree $3^{r+1}$. 
\end{theorem}
\begin{proof}
Let $\pi=1-\zeta_9$. An extension $L_1(\alpha^{1/3})/L_1$, with $\alpha\in L_1^{\times}$, 
is everywhere
unramified if and only if $(\alpha)$ is the cube of an ideal and
$\alpha$ is congruent to a cube mod $\pi^9$ (\cite[Exercise 9.3]{washington}).
Therefore, we need to show that $\varepsilon_0$ and $\gamma_i$
are cubes mod $\pi^9$.

We need a few preliminary calculations.
We have $\zeta_3=(1-\pi)^3=1-3\pi+3\pi^2-\pi^3$.
It follows that $\sqrt{-3}=1+2\zeta_3=3-6\pi+6\pi^2-2\pi^3$. Squaring yields
$$-3\equiv 4\pi^6-12\pi^3 \equiv \pi^6-3\pi^3 \pmod{\pi^{10}}.$$
Substituting this expression for the $-3$ in the right-hand side yields
$$
-3\equiv \pi^6+\pi^9\pmod{\pi^{10}}.
$$
Substituting this into the expression for $\sqrt{-3}$ yields
$$
\sqrt{-3}\equiv \pi^3-\pi^6-\pi^7+\pi^8\pmod{\pi^{10}}.
$$
Also,
$$
\zeta_3\equiv 1-\pi^3+\pi^7-\pi^8\pmod{\pi^{10}}.
$$
Let $a, b, c, \dots$ be integers. Using the above congruences, we obtain
$$
\left(1+a\pi+b\pi^2+c\pi^3+\cdots\right)^3\equiv 1+a\pi^3+b\pi^6-a\pi^7-
(a^2+b)\pi^8\pmod{\pi^9}.
$$
This is the general form of a cube that is congruent to 1 mod $\pi$.

Let $\gamma\in K_0^{\times}$ be such that $(\gamma)=I^3$ for some ideal $I$ of $K_0$.
By multiplying $I$ by a principal ideal, if necessary, we may assume that $I$
is prime to 3. Therefore, by multiplying $\gamma$ by a cube, we may assume 
that $\gamma$ is prime to 3.
Write $\gamma^2=u+v\sqrt{3d}$. It is easy to see that $u, v$ do not
have 3 in their denominators, so they may be regarded as elements of $\mathbf Z_3$.
Since we are assuming that $3$ splits in $F_0=\mathbf
Q(\sqrt{-d})$, we can regard $v\sqrt{-d}$ as an element of $\mathbf Z_3$
under an embedding $F_0\hookrightarrow \mathbf Q_3$. Hence  $v\sqrt{-d}$ is congruent
mod 3 to an integer $\ell$, and we can
regard $\gamma$ as an element of $\mathbf Z_3[\zeta_3]$. Moreover,
$\gamma^2\equiv 1\pmod{\sqrt{-3}}$. Therefore, $u=1+3k$ for some $k$, and
\begin{align*}
\gamma^2&\equiv 1+3k+\ell\sqrt{-3} \pmod{\pi^9\mathbf Z_3[\zeta_9]}\\
&\equiv 1+\ell\pi^3-(\ell+k)\pi^6-\ell\pi^7+\ell\pi^8.
\end{align*}
Since $(\gamma)=I^3$, we have $\text{Norm}(\gamma)=\pm \text{Norm}(I)^3$,
and therefore $\text{Norm}(\gamma)^2$ is a 6-th power, hence is congruent
to 1 mod 9. Therefore,
$$
1 \equiv \text{Norm}(\gamma)^2 = u^2-3dv^2\equiv 1+6k+3\ell^2\pmod{9\mathbf Z_3[\zeta_9]}.
$$
It follows that $\ell^2\equiv k\pmod 3$
and
$$
(1+\ell\pi-(\ell+k)\pi^2)^3\equiv \gamma^2\pmod{\pi^9\mathbf Z_3[\zeta_9]},
$$
so $\gamma^2$ is a cube mod $\pi^9$.

Since this calculation is valid for the completions at each of the two
primes above 3, we have proved that 
$\gamma^2$ is globally a cube mod $\pi^9$. This implies
that $L_1(\gamma^{1/3})/L_1$ is everywhere unramified.

Now suppose that 
$\varepsilon_0^{a_0}\gamma_1^{a_1}\cdots \gamma_r^{a_r}=\beta^3$
for some $\beta\in L_1$. By Lemmas~\ref{quadlemma} and \ref{cubelemma},
we may assume that $\beta\in K_0$.  Therefore,
$$
I_1^{a_1}\cdots I_r^{a_r}=(\beta)
$$
in $K_0$, which implies that $a_i\equiv 0\pmod 3$ for all $i$. Consequently,
$\varepsilon_0^{a_0}$ is a cube in $K_0$, so $a_0\equiv 0\pmod 3$.
Therefore, $\varepsilon_0, \gamma_1, \dots, \gamma_r$ are independent mod cubes
in $L_1$, so 
$$
L_1(\varepsilon_0^{1/3}, \gamma_1^{1/3}, \dots, \gamma_r^{1/3})/L_1
$$
is unramified of degree $3^{r+1}$. Since $\sigma$ acts trivially on these Kummer
generators, the Galois equivariance of the Kummer pairing implies that
$\sigma$ acts by inversion on the Galois group of this extension. Therefore,
$\text{rank}(A_1)=\text{rank}(\widetilde{A}_1^-)\ge r+1$.
\end{proof}

\begin{corollary} Let $d\equiv 2\pmod 3$ and let 
$r$ be the 3-rank of the class group of $K_0$. Then $\lambda\ge r+1$.
\end{corollary}
\begin{proof} This follows immediately from Theorem~\ref{unramthm} 
and Lemma~\ref{ranklemma}. \end{proof}

In particular, the corollary implies that if $3\mid h^+$ then $\lambda\ge 2$,
as was shown in Proposition~\ref{padicLprop} using $3$-adic $L$-functions.

When $3$ splits in $F_0$, we know that $\lambda\ge 1$. Correspondingly, in
this case, Theorem~\ref{unramthm} shows that we can always produce an explicit unramified
3-extension of $L_1$ using $\varepsilon_0^{1/3}$. 
This gives one explanation of the fact that $\lambda\ge 1$.
Note that we obtain an unramified extension of $L_1$ but not necessarily of $L_0$.
The latter case happens when $\varepsilon_0\equiv\pm 1 \pmod {3\sqrt{D}}$.

\section{Congruences}

We continue to restrict to the case that $d\equiv 2\pmod 3$, so $3$ splits 
in $F_0$. We complete
everything at one of the two primes of $L_1$ 
above 3. The completions of $K_0$ and $L_0$ 
yield $\mathbf Q_3(\zeta_3)$, and the completions of $K_1$ and $L_1$ yield
$\mathbf Q_3(\zeta_9)$. The element $\tau$ can be taken to be $\sigma_4\in
(\mathbf Z/9\mathbf Z)^{\times}\simeq \text{Gal}(\mathbf
Q_3(\zeta_9)/\mathbf Q_3)$ and $g$ becomes $\sigma_{-1}$, where
$\sigma_i(\zeta_9)=\zeta_9^i$. Recall that $\pi=1-\zeta_9$. 

\begin{lemma}\label{pi11lemma} Let $\varepsilon$ be a $\pi$-adic unit in $\mathbf Q_3(\zeta_9)$ such that 
$\sigma_{-1}(\varepsilon)= \varepsilon^{-1}$
times a cube. Then 
$$\varepsilon\equiv \pm\zeta_9^a \pmod {\pi^{5}}$$
for some integer $a$, and
$$
\varepsilon^3\equiv \pm\zeta_3^a \pmod {\pi^{11}}.
$$
\end{lemma}
\begin{proof}
As a preliminary result, note that the cube of a $\pi$-adic
unit is congruent to $\pm 1\pmod {\pi^3}$ ({\it proof:} compute the cube 
of $a+b\pi+\cdots$).
Also, $\zeta_9^a=(1-\pi)^a\equiv 1-a\pi\pmod{\pi^2}$. Therefore, 
if $\pm\varepsilon=1+a\pi+\cdots$, then $\pm\varepsilon\zeta_9^a\equiv 
(1+a\pi)(1-a\pi)\equiv 1\pmod {\pi^2}$.

Since
$$
\sigma_{-1}(\pi)=1-(1-\pi)^{-1}\equiv -\pi\pmod{\pi^2},
$$
it follows that $\sigma_{-1}(\pi^j)\equiv (-1)^j\pi^j\pmod{\pi^{j+1}}$.

Suppose that $\sigma_{-1}(\varepsilon)\equiv \varepsilon^{-1}$ mod cubes. Write 
$\pm\varepsilon\zeta_9^a=1+b\pi^2+\cdots$.
Since $\sigma_{-1}(\zeta_9)=\zeta_9^{-1}$, we have
\begin{align*}
1+(-1)^2 b\pi^2&\equiv \sigma_{-1}(1+b\pi^2)\equiv\sigma_{-1}(\pm\varepsilon\zeta_9^a)= 
(\pm\varepsilon^{-1}\zeta_9^{-a})\times(\text{cube})\\
&\equiv \pm(1-b\pi^2)\pmod{\pi^3}.
\end{align*}
Therefore, $b\equiv 0\pmod 3$, so
$$
\pm\varepsilon\zeta_9^a=1+c\pi^3+\cdots.
$$
Since $\zeta_3\equiv 1-\pi^3\pmod {\pi^4}$, we have
$$\pm\varepsilon\zeta_9^a\zeta_3^{c}\equiv 1\pmod {\pi^4}.
$$
Cubing this yields
$$
\pm\varepsilon^3\zeta_3^a\equiv 1\pmod {\pi^{10}}.
$$
We have therefore proved that the cube of a $\pi$-adic unit satisfying
$\sigma_{-1}(\varepsilon) \equiv \varepsilon^{-1}$ mod cubes is congruent to a sixth root
of unity mod $\pi^{10}$. 

Now write $\pm\varepsilon\zeta_9^e=1+f\pi^4+\cdots$, for some $e,f$. Then
$$
\varepsilon^{1+\sigma_{-1}} = 1+2f\pi^4+\cdots.
$$
Since this is assumed to be a cube,
it is congruent to $\pm \zeta_3^a\equiv \pm(1-a\pi^3) \pmod{\pi^6}$ for some $a$.
Therefore, $f\equiv 0\pmod 3$, so $\pm\varepsilon\zeta_9^e\equiv 1\pmod {\pi^5}$.
This yields the first part of the lemma. Cubing yields the second part.
\end{proof}

When 3 splits in $F_0$, we know that $\lambda\ge 1$. The following gives criteria
for $\lambda\ge 2$, and gives an algebraic proof of Proposition 1(a), which was proved
using $3$-adic $L$-functions.

\begin{theorem}\label{main2thm} Let $d\equiv 2\pmod 3$ and assume that $3\nmid h^+$. 
The following are equivalent:
\newline
(a) $\lambda\ge 2$. 
\newline
(b) $L_1(\varepsilon_1^{1/3})/L_1$ is unramified.
\newline
(c) $\log_3\varepsilon_0\equiv 0\pmod{\pi^{10}}$.
\newline
(d) $\log_3\varepsilon_0\equiv 0\pmod{\pi^{15}}$.
\end{theorem}
\begin{proof}
The $3$-adic class number formula shows that $(\log_3\varepsilon_0)/\sqrt{D}
\in \mathbf Q_3$, and the $\pi$-adic valuation
of an element of this field is a multiple of 6. Since the $\pi$-adic valuation
of $\sqrt{D}$ is 3, the equivalence of (c) and (d) follows.

Lemma 2 and Theorem~\ref{mainthm} show that (b) implies (a). Proposition 1 shows
that (a) implies (c).

We now prove the equivalence of (b) and (c). Let $\varepsilon_2$ be as in Theorem 1.
By Lemma~\ref{pi11lemma}, we can write 
$$
\varepsilon_2\equiv \pm \zeta_9^{a}(1+a_5\pi^5+a_6\pi^6+a_7\pi^7+a_8\pi^8+a_9\pi^9)
\pmod {\pi^{10}}.
$$

A calculation shows that
\begin{align*}
\tau(\pi^5)&\equiv \pi^5-\pi^7+\pi^8+\pi^9\pmod{\pi^{10}}\\
\tau(\pi^6)&\equiv \pi^6\\
\tau(\pi^7)&\equiv \pi^7+\pi^9\\
\tau(\pi^8)&\equiv \pi^8\\
\tau(\pi^9)&\equiv \pi^9.
\end{align*}
This yields
$$
\varepsilon_0=\varepsilon_2^{1+\tau+\tau^2}
\equiv \pm\zeta_3^{a} (1+2a_5\pi^9)\pmod{\pi^{10}}.
$$
Therefore, $\log_3\varepsilon_0\equiv 2a_5\pi^9\pmod {\pi^{10}}$.

Suppose now that $\log_3\varepsilon_0\equiv 0\pmod {\pi^{10}}$. Then
$a_5\equiv 0\pmod 3$. Therefore,
\begin{align*}
\varepsilon_1&=\varepsilon_2^{1-\tau}\\
&\equiv (1+a_6\pi^6+a_7\pi^7+a_8\pi^8+a_9\pi^9)^{1-\tau}\\
&\equiv 1-a_7\pi^9\pmod{\pi^{10}}.
\end{align*}
This means that $\varepsilon_1$ is congruent to a cube mod $\pi^9$,
so $L_1(\varepsilon_1^{1/3})/L_1$ is unramified by \cite[Exercise 9.3]{washington}.

Conversely, suppose that $L_1(\varepsilon_1^{1/3})/L_1$ is unramified,
so $\varepsilon_1$ is congruent to a cube mod $\pi^9$. 
Although we already know from Lemma 2 and Proposition 1 that $\log_3\varepsilon_0
\equiv 0\pmod {\pi^{10}}$, we prove this algebraically.

Since $\zeta_3=\zeta_9^3$,
we have that 
$\zeta_3^{a}\varepsilon_1$ is congruent to a cube mod $\pi^9$.
The $\pi$-adic expansion of $\varepsilon_2$ yields
\begin{align*}
\zeta_3^{a}\varepsilon_1&=\zeta_3^{a}\varepsilon_2^{1-\tau}\\
&\equiv (1+a_5\pi^5+\cdots)^{1-\tau} \pmod{\pi^9}\\
&\equiv 1+a_5\pi^7-a_5\pi^8.
\end{align*}
By Lemma~\ref{pi11lemma}, this must be congruent to $\zeta_3^c$ mod $\pi^9$
for some $c$. Since $\zeta_3^c\equiv 1-c\pi^3\pmod {\pi^4}$, we must have $\zeta_3^c=1$,
which implies that $a_5\equiv 0\pmod 3$. Therefore,
$\log_3\varepsilon_0\equiv 0\pmod {\pi^{10}}$.
This completes the proof.\end{proof}

\begin{corollary}\label{lam2cor} Assume that $d\equiv 2\pmod 3$. Then $\lambda\ge 2$
if and only if $\text{rank}(A_1)\ge 2$.
\end{corollary}
\begin{proof}
When $3\nmid h^+$, this follows from the equivalence of (a) and (b) in the theorem,
plus Theorem~\ref{mainthm}(d). When $3\mid h^+$, 
Theorem~\ref{unramthm} implies that $\text{rank}(A_1)\ge 2$
and hence that $\lambda\ge 2$. 
\end{proof}

Therefore, we can see whether or not  $\lambda\ge2$ already at $F_1$. 
When $\lambda\ge 2$, we can obtain more information about Kummer generators.

\begin{proposition} Assume that $d\equiv 2\pmod 3$, that $3\nmid h^+$, and that
$\lambda\ge 2$.\newline
(a) $A_2$ is not an elementary 3-group. \newline
(b) The extension $L_2(\varepsilon_0^{1/9})/L_2$ is everywhere unramified.\newline
(c) $3\mid h^-$ if and only if $L_1(\varepsilon_0^{1/9})/L_1$ is unramified.
\end{proposition}
\begin{proof} Choose a prime of $\overline{\mathbf Q}$ above 3 and work in
the completion at this
prime.
Theorem~\ref{main2thm} says that $\log_3\varepsilon_0\equiv 0\pmod{\pi^{15}}$,
so $\log_3(\varepsilon_0^{1/3})\equiv 0\pmod {\pi^9}$. Let 
$$
y=\exp(\log_3(\varepsilon_0^{1/3}))\in \mathbf Q_3(\zeta_3).
$$
Then $y\equiv 1\pmod {\pi^9}$ and $\log_3 y=\log_3(\varepsilon_0^{1/3})$.
The kernel of $\log_3$ consists of (integral and fractional)
powers of 3 times roots of unity.
Therefore, $y=\pm \zeta_9^a\varepsilon_0^{1/3}$ for some integer $a$, and 
$\varepsilon_0^{1/3}\equiv \pm\zeta_9^{-a}\pmod{\pi^9}$, so
$\varepsilon_0^{1/3}$ is congruent to a cube in $\mathbf Q_3(\zeta_{27})$ mod $\pi^9$,
so the extension
$\mathbf Q_3(\zeta_{27},\varepsilon_0^{1/9})/\mathbf Q_3(\zeta_{27})$
is unramified. But this is the completion of the extension
$L_2(\varepsilon_0^{1/9})/L_2$ at any of the primes above 3.
Since this extension is unramified at all other primes because 
$\varepsilon_0$
is a unit, the extension is everywhere unramified. This proves (a) and (b).

We now prove (c).
The proof of Scholz's theorem 
(see \cite[Theorem 10.10]{washington}) yields that
$$
3\text{-rank of class group of }F_0=\delta,
$$
where $\delta=1$ if $L_0(\varepsilon_0^{1/3})/L_0$ is unramified and $\delta=0$
otherwise. The relation $\pm \zeta_9^a\varepsilon_0^{1/3}=y\equiv 1\pmod{\pi^9}$ 
shows that, in any
completion at a prime above 3, $\varepsilon_0\equiv\pm \zeta_3^{-a}\pmod {\pi^{15}}$
for some $a$ that possibly depends on the choice of completion. Since
$(1+b_1\sqrt{-3}+\cdots)^3\equiv 1\pmod {3\sqrt{-3}}$,
we see that $\zeta_3^a$ is a cube of an element
of $L_0$ mod $\pi^9$
if and only if $a\equiv 0\pmod 3$. However, since the only Galois conjugates
of $\varepsilon_0$ are $\varepsilon_0$ and $\varepsilon_0^{-1}$, if $\varepsilon_0\equiv
1\pmod {\pi^9}$ at one completion at a prime above 3, then this holds for all 
such completions, hence holds globally. In other words, if we have $a\equiv 0\pmod 3$
when working in one completion,
then this holds in all the other completions.

Suppose now that $3|h^-$. Then $\delta=1$, which means that we must have 
$\varepsilon_0\equiv 1\pmod {\pi^9}$ and $a\equiv 0\pmod 3$. Therefore,
$\varepsilon_0^{1/3}\equiv \pm \zeta_9^{-a} \pmod{\pi^9}$ means that $\varepsilon_0^{1/3}$
is congruent to a cube mod $\pi^9$. Therefore, $L_1(\varepsilon_0^{1/9})/L_1$
is unramified.

Conversely, suppose $L_1(\varepsilon_0^{1/9})/L_1$
is unramified. Then $\varepsilon_0^{1/3}$ is congruent to a cube mod $\pi^9$.
This means that $a\equiv 0\pmod 3$, so $\varepsilon_0\equiv 1\pmod {\pi^9}$.
Therefore, $L_0(\varepsilon_0^{1/3})/L_0$ is unramified, so $3|h^-$.
\end{proof}

When $3|h^-$, part (c) implies that $A_1$ is not elementary.
We can also prove this as follows. Suppose that $A_1$
is elementary. Since $3\nmid h^+$, the rank of $A_0$ is at most 1.
Since $A_0$ is a quotient of $A_1$, it is also elementary,
hence of order 3.
Let $3^{e_i}=|A_i|$. Gold \cite{gold} proves that if $d\equiv 2\pmod 3$ 
and if $e_1-e_0\le 2$ then $\lambda=e_1-e_0$. By Theorem~\ref{mainthm},
the rank of $A_1$ is at most 3, so $e_1\le 3$. Therefore $e_1-1=e_1-e_0\le 2$,
so $\lambda=e_1-1$. But Lemma 2 says that $e_1\le \lambda$, so we have a contradiction.
This proves that $A_1$ is not elementary. Note that we did not need any assumption
on $\lambda$ for this argument.

\section{Another point of view}

Gross-Koblitz \cite{gross} show that if 3 splits in $F_0$ as 
$\mathfrak p\overline{\mathfrak p}$, then
$$
L_3'(0, \chi)=2\log_3\alpha,
$$
where $(\alpha)=\mathfrak p^{h^-}$ and the logarithm is taken in the
$\overline{\mathfrak p}$-adic completion. Therefore,
$$
2\log_3\alpha = f'(0)\log_3(1+3) \equiv (a_1)(3) \pmod 9.
$$
This yields the following:
\begin{lemma}\label{gklemma} $\lambda\ge 2$ if and only if $\log_3 \alpha\equiv 0\pmod 9$.
\end{lemma}

Moreover, the $3$-adic analytic class number formula implies that
$$
3a_1\equiv f(3) = \frac{2h^+\log_3 \varepsilon_0}{\sqrt{D}}  \pmod 9,
$$
so we obtain the interesting relation
$$
\log_3(\alpha)\equiv \frac{h^+\log_3 \varepsilon_0}{\sqrt{D}}  \pmod 9.
$$
Note that $\alpha$ is a $3$-unit that lies in the $+$ component
for the action of $g$ and the $-$ component for the action of $\sigma$ 
on $3$-units mod
powers of $3$. The units $\varepsilon_i$ that we worked with in previous sections
are in the components
of the parities opposite from $\alpha$. 

The fact that Theorem 4 is based on the 3-adic $L$-function
at both 0 and 1 means that it is natural that it involves both $F_0$ and
$K_0$, which are related to each other through reflection theorems,
hence through the relations between Kummer generators and class groups.
In the present situation, everything involves only the 3-adic $L$-function
at 0, hence involves only $F_0$, not $K_0$. 
We prove the following analogue of Theorem 4. It is
interesting to note that the higher genus groups in the proof
of Theorem 1 are replaced by ``higher ambiguous groups''
of ideal classes, namely ideal classes that are annihilated by some
power of $1-\tau$.

\begin{theorem}\label{gkprop} The following are equivalent:\newline
(a) $\lambda\ge 2$\newline
(b) the 3-rank of $A_1$ is greater than or equal to 2\newline
(c) $\alpha$ is a norm for $F_1/F_0$.
\end{theorem}
\begin{proof}
Lemma~\ref{ranklemma} shows that (b) implies (a). 

Assume (a). Lemma~\ref{gklemma} implies that $\log_3\alpha\equiv 0\pmod 9$,
where the logarithm is taken in the $\overline{\mathfrak p}$-adic completion of $F_0$.
Let $\gamma=\exp((1/3)\log_3(\alpha))$. Then $\log_3(\gamma^3)=
\log_3(\alpha)$,
so $\gamma^3$ and $\alpha$ differ by a root of unity in the completion of $F_0$ at $\overline{\mathfrak p}$.
Since $3$ splits in $F_0$, this completion is $\mathbf Q_3$, 
whose only roots of unity are $\pm 1$.
Therefore, $\alpha$ is a local cube, hence a local norm at $\overline{\mathfrak p}$ for $F_1/F_0$. Since $F_1/F_0$
is unramified at all primes not above 3, $\alpha$ is a local norm for all places except possibly 
$\mathfrak p$. By the product formula for the norm residue symbol, $\alpha$ is a norm
also at $\mathfrak p$. Since $F_1/F_0$ is cyclic, Hasse's norm theorem says that $\alpha$ is a global
norm from $F_1$. (A similar argument appears in \cite{gold}.)

Assume (c). Write $\text{Norm}_{F_1/F_0}(\beta)=\alpha$. 
Then the ideal norm of $(\beta)$ is $(\alpha)=\mathfrak p^{h^-}$.
Let $\mathfrak P$ be the prime above $\mathfrak p$, so $\text{Norm}_{F_1/F_0}(\mathfrak P)=\mathfrak p$.
Therefore, the $\mathfrak P$-adic valuation of $\beta$ is $h^-$. 
Since $\text{Norm}_{F_1/F_0}((\beta)\mathfrak P^{-h^-})=(1)$,
we have $(\beta)=\mathfrak P^{h^-}J^{1-\tau}$ for some ideal $J$ of $F_1$. 

Let $\text{ord}(B)$ denote the order of an ideal $B$ in the class group
of $F_1$. Since the natural map from the class group of $F_0$ to the class group
of $F_1$ is injective (see \cite[Proposition 13.26]{washington}), if $B$ is an ideal of $F_0$,
then its order in the class group of $F_0$ equals its order in the class group
of $F_1$, so we need to consider only the order in $F_1$.

\begin{lemma}\label{3lemma} 
$\text{ord}(\mathfrak P)=3\text{ord}(\mathfrak p)$.
\end{lemma}
\begin{proof} (cf.~\cite[proof of Proposition 2]{greenberg}) If $3|\text{ord}(\mathfrak P)$,
then the lemma follows immediately from the fact that $\mathfrak P^3
=\mathfrak p$. So we need to show that $3|\text{ord}(\mathfrak P)$.

Suppose $\mathfrak P^b$
is principal for some $b>0$. Then $(\mathfrak P/\overline{\mathfrak P}_1)^b
=(\gamma)$ for some $\gamma\in F_1$. Let $\delta=\gamma/\overline{\gamma}$. Then
$$
(\delta)=(\mathfrak P/\overline{\mathfrak P}_1)^{2b}.
$$ 
Since this ideal is fixed by
$\tau$, it follows that $\delta^{1-\tau}$ is a unit. Since it has absolute value 1
at all embeddings into $\mathbf C$, it is a root of unity, hence $\pm 1$.
Since the norm from $F_1$ to $F_0$ of $\delta^{1-\tau}$ is 1, we must have
$\delta^{1-\tau}=+1$. Therefore, $\delta\in F_0$. 
Since $2b=v_{\mathfrak P}(\delta)=3v_{\mathfrak p}(\delta)$,
we have $3|b$. Therefore, $3|\text{ord}(\mathfrak P)$. \end{proof}

Suppose $\mathfrak P^bJ^c$ is principal for some $b, c$. 
Applying $1-\tau$ yields
that $J^{(1-\tau)c}$ is principal, which implies that $\mathfrak P^{ch^-}$
is principal. 
Write $h^-=3^n v$ with $3\nmid v$.

{\bf Case 1.} Suppose $3^n|\text{ord}(\mathfrak p)$. Then 
$3^{n+1}|\text{ord}(\mathfrak P)$. Therefore, $3|c$. It follows that
the subgroup of the ideal class group generated by $\mathfrak P$
and $J$ has order at least 3 times the order of the subgroup generated
by $\mathfrak P$, hence has order at least $3^{n+2}$. Therefore,
$|A_1|\ge 9|A_0|$. The following lemma implies that $A_1$ cannot be cyclic
(cf. \cite{moriya}).

\begin{lemma} If $A_1$ is cyclic, then $|A_1|=3|A_0|$.
\end{lemma}
\begin{proof}
Write $A_1\simeq \mathbf Z/3^m\mathbf Z$, where $m\ge 1$ by Lemma~\ref{3lemma}. 
Then $\tau\in\text{Aut}(A_1)$
has order 3, hence corresponds to multiplication by $1+3^{m-1}x$ for
some $x$. Therefore, $1+\tau+\tau^2\equiv 3\pmod {3^m}$. 
The map $1+\tau+\tau^2$ is the endomorphism of $A_1$
given by the norm followed by the natural map from $A_0$ to $A_1$.
Since $F_1/F_0$
is totally ramified, the norm map from $A_1$ to
$A_0$ is surjective.
Since $A_0$ lies in the minus component with respect to complex conjugation,
the map $A_0\to A_1$ is injective. It follows that
so $A_0\simeq 3\mathbf Z/3^m\mathbf Z$.\end{proof}

{\bf Case 2.} Suppose $3^n\nmid \text{ord}(\mathfrak p)$. Then 
$\mathfrak P^{h^-}$ is principal, so the subgroup of $A_1$
generated by $A_0$ and $\mathfrak P$ has exponent at most
$h^-$. 
\begin{lemma} Every ideal class of $A_1$ that is fixed by $\tau$
contains an ideal of the form $I\mathfrak P^a$, with $I$
an ideal from $F_0$.\end{lemma}
\begin{proof} Let $B$ represent a fixed ideal class, so $B^{1-\tau}=(\gamma)$
for some $\gamma\in F_1$. Then the norm from $F_1$ to $F_0$ of $\gamma/\overline{\gamma}$
is a unit and has absolute value 1 at all places, hence is a root of unity.
The only roots of unity in $F_1$ are $\pm 1$, so $\text{Norm}
(\gamma^2/\overline{\gamma}^2)=1$,
which means that $\gamma^2/\overline{\gamma}^2=\delta^{1-\tau}$ for some $\delta\in F_1$.
Therefore, $B^2/\delta$ is fixed by $\tau$, so it is of the form 
$\mathfrak P^aI$ with $I$ an ideal from $F_0$. Since 
the class of $B$ has 3-power order,
the class of $B$ also contains an ideal of this form. \end{proof} 

The order of the subgroup of the class group of $F_1$ that is fixed
by $\tau$ is $3h^-$ (see \cite[Ia, Satz 13]{hasse}). Therefore, this subgroup
has order greater than its exponent, so it is noncyclic. It follows
that $A(F_1)$ is noncyclic. 

This completes the proof that (c) implies (b). \end{proof}

The theorem has the interesting corollary that when $3$ 
splits in $F_0=\mathbf Q(\sqrt{-d})$, we can determine whether
$\lambda \ge 2$ by looking at the first level $F_1$ 
of the $\mathbf Z_3$-extension. Namely, we have $\lambda\ge 2$
if and only if $\text{rank }A_1\ge 2$. This was also obtained 
from the results proved using Kummer 
generators in Section 8.

When $3\nmid h^+$ but $3\mid h^-$, the condition that $\alpha$ is a local cube in the completion at $\overline{\mathfrak p}$ can
be strengthened to saying that $\alpha$ is a global cube.

\begin{proposition}\label{gkprop2} Assume that $3$ splits in $F_0$, that $3\nmid h^+$, 
and that $3\mid h^-$. Then
$\lambda\ge 2$ if and only if $\alpha$ is a cube in $F_0$. \end{proposition}
\begin{proof}
Assume that $\lambda\ge 2$. By the proof of Theorem~\ref{gkprop}, $\alpha$ is a cube in the completion of $F_0$ at
$\overline{\mathfrak p}$. Therefore, $L_0(\alpha^{1/3})/L_0$ is unramified at $\overline{\mathfrak p}$.
This implies that $L_0(\overline{\alpha}^{1/3})/L_0$ is unramified at $\mathfrak p$.

Since 
$$
(\alpha\overline{\alpha})=(\mathfrak p\overline{\mathfrak p})^{h^-}=(3)^{h^-},
$$
$\alpha\overline{\alpha}=3^{h^-}$ (the unit must be $+1$ since both sides are positive), which is a cube since $3\mid h^-$.
Therefore,
$$
L_0(\alpha^{1/3})=L_0(\overline{\alpha}^{1/3}).
$$
It follows that both $\mathfrak p$ and $\overline{\mathfrak p}$ are unramified in $L_0(\alpha^{1/3})/L_0$.
Since only primes above 3 can be ramified in this extension, the extension must be everywhere unramified.
The fact that $\alpha\overline{\alpha}$ is a cube says that $\sigma$ (= complex conjugation) acts by
inversion on $\langle \alpha\rangle$ mod cubes. Also, $\sigma$ 
acts by inversion on $\mu_3$. The Galois
equivariance of the Kummer pairing implies that $\sigma$ acts trivially on $\text{Gal}(L_0(\alpha^{1/3})/L_0)$,
which means that this extension corresponds to part of the ideal class of $K_0$. Since $3\nmid h^+$, this
extension must be trivial, so $\alpha$ is a cube in $L_0$. By Lemma~\ref{cubelemma}, $\alpha$ is a cube in $F_0$.

Conversely, if $\alpha$ is a cube in $F_0$, say $\alpha=\beta^3$, then $\log_3\alpha=3\log_3\beta\equiv 0\pmod 9$,
since $\log_3 x\equiv 0\pmod 3$ for all $x\in \mathbf Q_3$. Therefore, $\lambda\ge 2$ by Lemma~\ref{gklemma}.
\end{proof}

The conditions that $3\nmid h^+$ and $3\mid h^-$ are essential in
Proposition~\ref{gkprop2}. When $d=35$, we have $h^-=h^+=2$
and $\alpha=(1+\sqrt{-35})/2$. Therefore, $\log_3\alpha\equiv 0\pmod 9$,
so $\lambda\ge 2$. However, $\alpha$ is not a cube.
When $d=107$, we have $h^-=h^+=3$ and $\alpha=(1+\sqrt{-107})/2$.
Therefore, $\log_3\alpha\equiv 0\pmod 9$, so $\lambda\ge 2$.
Again, $\alpha$ is not a cube.

\begin{corollary}  Assume 3 splits in $F_0$, that $3\nmid h^+$, and that
$3\mid h^-$.  Then
the rank of $A_1$ is greater than or equal to 2 if and only if $\alpha$ is a cube.
\end{corollary}

\begin{proof} 
This follows from combining Theorem \ref{gkprop} and Proposition \ref{gkprop2}.
However, proving that $\alpha$ is a cube when the rank is at least 2  
used 3-adic $L$-functions (in the step that $\lambda\ge 2$ implies
that $\alpha$ is a local cube). 
We can also give an algebraic proof of this result.

Since $\text{rank } A_1\ge 2$, the two primes above 3 cannot remain
fully inert in the Hilbert 3-class field $H$ of $F_1$.  Thus the splitting field $S$
for one of the primes above 3 is a nontrivial extension of $F_1$.  As
$\sigma\in\text{Gal}(F_1/\mathbf B_1)$ acts as $-1$ 
on $\text{Gal}(H/F_1)$,
this splitting field is Galois over $\mathbf B_1$.  Thus $S$ is the splitting field
for both
primes above 3.  This makes $S$ Galois over $F_0$ since $S$ 
is the maximal subextension
of $H/F_1$ in which the primes above 3 split.

Since $\tau$ has order 3 and acts on the 3-group $\text{Gal}(S/F_1)$,
it has a non-trivial quotient on which it acts trivially.   
Therefore $S_0$, the maximal abelian extension of $F_0$ contained in $S$,
is a nontrivial extension of $F_1$.  Let $N$ be a degree 3 extension of $F_1$
contained in $S_0$.  Then $N/F_0$ cannot be a cyclic degree 9 extension
because the ramification is at the bottom.  Let $N'$ be the 
inertia field for $\mathfrak p$,
one of the primes above 3 in $F_0$.  Then $N'/F_0$ is a cyclic degree 3 extension
in which $\mathfrak p$ is unramified. Since the prime above
$\mathfrak p$ splits in $N/F_1$
and $\mathfrak p$ ramifies in $N/N'$, it follows that $\mathfrak p$ splits in $N'/F_0$.

Let $\gamma$ be a Kummer generator for $L_0 N'/L_0$.
Since $\sigma$ acts by inversion on $\text{Gal}(L_0N'/L_0)$ and by inversion
on $\mu_3$, the Galois equivariance of the Kummer pairing implies that
$\sigma$ acts trivially on $\gamma$ mod cubes. By Lemma~\ref{cubelemma},
we may assume that $\gamma\in K_0$. 
Let $\mathfrak p_0$ be the prime of $K_0$ above 3 and let 
$a=v_{\mathfrak p_1}(\gamma)$. Since the extension $L_0 N'/L_0$
is unramified at $\mathfrak p$, we must have
$a\equiv 0\pmod 3$.  Moreover, the extension is unramified at all primes
not above 3, so we have $(\gamma)=I^3$ for some ideal $I$ of $K_0$.
Since $3\nmid h^+$, the ideal $I$ is principal, which means that $\gamma$
is a power of $\varepsilon_0$ times a cube.
Therefore, $N'L_0=L_0(\varepsilon_0^{1/3})$, 
and $N'/F_0$ is the cyclic
degree 3 subextension of the Hilbert class field of $F_0$ that lifts
to this subextension of the Hilbert class field of $L_0$.
Since $\mathfrak p$ splits in $N'/F_0$,
the image of $\mathfrak p$ under the Artin map is contained in the subgroup
of order $h'/3$ that fixes $N'$. Therefore, the image of 
$\mathfrak p^{h^-/3}$ is
trivial, so $\mathfrak p^{h^-/3}$ is principal. The only units in $F_0$ are
$\pm 1$, so $\mathfrak p^{h^-} =(\alpha)$ implies that $\alpha$ is a cube in $F_0$, as desired.
\end{proof}

\section{The inert case}

Suppose that 3 is inert in $F_0$ and that $3|h^-$. Then Proposition 1 says that 
$$\lambda\ge 2\iff h^-\equiv \frac{h^+\log_3 \varepsilon_0 }{\sqrt{D}} \pmod 9.$$
When $3\| h^-$, this is a much more subtle situation than the case when $3$ splits. 
We are not simply asking that a sufficiently high power of 3 divide $\log_3
\varepsilon_0$, for example. Instead, we are asking that the nonzero congruence
class of $h^-$ mod 9 match the congruence class of an expression involving
$h^+$ and $\log_3\varepsilon_0$. In contrast to most situations, what is relevant 
is not
simply the power of 3 that divides the numbers, but rather the congruence classes 
mod 3 that are obtained
when the numbers are divided by suitable powers of 3 (that is, not just the ``order
of vanishing'' but also the ``coefficient of the leading term'').
We hope to treat this situation in the future (but we make no promises).

\smallskip

\noindent
{\bf Examples.}

When $d=31$, we have $h^+=1, h^-=3$, and 
$\varepsilon_0=(29+3\sqrt{93})/2\equiv 1+6\sqrt{93}\pmod{9}$.
Therefore, 
$$
h^-=3\not\equiv \frac{h^+\log_3\varepsilon_0}{\sqrt{93}}\equiv 6\pmod 9,
$$
so $\lambda=1$.

When $d=211$, we have $h^+=1, h^-=3$, and 
$\varepsilon_0\equiv 1+3\sqrt{633}\pmod{9}$.
Therefore, 
$$
h^-=3\equiv \frac{h^+\log_3\varepsilon_0}{\sqrt{633}}\pmod 9,
$$
so $\lambda\ge 2$.

When $d=244$, we have $h^+=2, h^-=6$, and 
$\varepsilon_0\equiv 1\pmod{9}$.
Therefore, 
$$
h^-=6\not\equiv \frac{h^+\log_3\varepsilon_0}{\sqrt{732}}\equiv 0\pmod 9,
$$
so $\lambda=1$.

\end{document}